\documentclass[letterpaper, 10 pt, conference]{ieeeconf}  

\IEEEoverridecommandlockouts                              
\usepackage{amsfonts,amsmath,amssymb,bm}
\usepackage{color,cases,float}
\usepackage{graphicx}
\usepackage{hyperref}
\usepackage{psfrag}
\usepackage{eucal}
\usepackage{algorithm}
\usepackage{overpic}
\usepackage{subfigure}
\usepackage{url}
\usepackage[all,cmtip]{xy}
\usepackage{algorithm}
\usepackage[noend]{algpseudocode}
\usepackage[inkscapelatex=false]{svg} 

\definecolor{darkblue}{rgb}{0.0,0.0,0.6}
\hypersetup{colorlinks,breaklinks,
	linkcolor=darkblue,urlcolor=darkblue,
	anchorcolor=darkblue,citecolor=darkblue}
\newtheorem{assumption}{Assumption}
\newtheorem{definition}{Definition}
\newtheorem{lem}{Lemma}

\newtheorem{rem}{Remark}
\newtheorem{prop}{Proposition}
\newtheorem{theorem}{Theorem}

\makeatletter
\newcommand{\setassumptiontag}[1]{
  \let\oldtheassumption\theassumption
  \renewcommand{\theassumption}{#1}
  \g@addto@macro\endassumption{
    \addtocounter{assumption}{-1}
    \global\let\theassumption\oldtheassumption}
  }
\makeatother

\newcommand{\E}{\mathbb{E}}

\newcommand{\R}{\mathrm{Re}}

\def\blab{\boldsymbol{\lambda}}

\def\R{\mathbb{R}}

\def\zb{\mathbf{z}}
\def\bx{\boldsymbol{x}}

\def\eb{\mathbf{e}}
\def\nb{\mathbf{n}}

\def\pb{\mathbf{p}}

\def\bb{\mathbf{b}}

\def\yb{\mathbf{y}}

\def\rb{\boldsymbol{r}}
\def\lb{\boldsymbol{l}}

\def\R{\mathbb{R}}

\def\zb{\mathbf{z}}
\def\xb{\boldsymbol{x}}

\def\eb{\mathbf{e}}

\def\pb{\mathbf{p}}
\def\qb{\mathbf{q}}

\def\bb{\boldsymbol{b}}

\def\Mb{\boldsymbol{M}}

\def\BibTeX{{\rm B\kern-.05em{\sc i\kern-.025em b}\kern-.08em
		T\kern-.1667em\lower.7ex\hbox{E}\kern-.125emX}}

	\title{Learning Generalized Nash Equilibria in Non-Monotone Games with Quadratic Costs}
	\author{Tatiana Tatarenko, Lucas Wey Hacker
		\thanks{The authors are with the Department of Control Theory and Intelligent Systems, TU Darmstadt, Germany. E-mails: {\tt tatiana.tatarenko@tu-darmstadt.de},
{\tt lucas.hacker@tu-darmstadt.de}.}}

\begin{document}

\maketitle
\thispagestyle{empty}
\pagestyle{empty}

\begin{abstract}
We study generalized Nash equilibrium (GNE) problems in games with quadratic costs and individual linear equality constraints. Departing from approaches that require strong monotonicity and/or shared constraints, we reformulate the KKT conditions of the (generally non-monotone) games into a tractable convex program whose objective satisfies the Polyak–Łojasiewicz (PL) condition. This PL geometry enables a distributed gradient method over a fixed communication graph with global geometric (linear) convergence to a GNE. When gradient information is unavailable or costly, we further develop a zero-order fully distributed scheme in which each player uses only local cost evaluations and their own constraint residuals. With an appropriate step size policy, the proposed zero-order method converges to a GNE, provided one exists, at rate 
$O(1/t)$.
\end{abstract}

\section{Introduction}\label{sec:intro}
Generalized Nash Equilibrium (GNE) problems, where players' strategy sets are coupled, arise ubiquitously in diverse fields ranging from electricity markets to network games~\cite{hobbs2000strategic, scutari2012real}. The challenge lies in finding solutions in these games, especially when players have limited information or computational resources. 

Recent work has focused on developing iterative algorithms for learning GNE, often relying on strong assumptions like monotonicity of the game mapping~\cite{LinEqGram, jordan2023first, meng2023linear, TatKam2025arxiv, suli}. However, many real-world GNE problems lack 
such properties, demanding more sophisticated  solution techniques.
Furthermore, a significant portion of the existing literature on GNE focuses on games with shared constraints, where all players are subject to the same set of constraints~\cite{LinEqGram, meng2023linear, TatKam2025arxiv, suli}. 

Motivated by these limitations, our work addresses GNE problems with non-monotone game mappings and individual coupling constraints. This allows for modeling scenarios where players face distinct limitations or requirements, reflecting the heterogeneity often present in real-world systems. Specifically, we focus on games with quadratic cost functions and linear equality coupling constraints, a structure often arising in engineering and economic models \cite{budish2022flow, facchinei2007generalized, gould2001equality, ye2020extremum}.
By reformulating the corresponding Karush-Kuhn-Tucker conditions for GNE in games with such a specific structure, we derive a tractable convex optimization problem suitable for distributed solution. This reformulation is critical because, unlike existing approaches, it avoids relying on (strong) monotonicity assumptions. Moreover, we demonstrate that the objective function fulfills a so-called Polyak-Łojasiewicz (PL) condition, enabling the efficient application of well-established gradient-based algorithms under a defined communication topology within the system~\cite{PL_distr}. Consequently, we guarantee geometric convergence for calculating GNE in the games under consideration. It is worth noting that the work \cite{LinEqGram} presents a communication-based procedure achieving geometrically fast convergence to a GNE in games with strongly monotone pseudo-gradients and shared linear equality constraints. The paper \cite{meng2023linear} in turn proves an analogous result for strongly monotone games with shared linear inequality constraints, subject to a specific rank condition required for fast convergence. Thus, our work extends the existing literature by demonstrating geometrically fast distributed convergence in games with non-monotone pseudo-gradients and individual coupling constraints.     

%

Finally, unlike most existing approaches that rely on gradient information~\cite{LinEqGram, jordan2023first, meng2023linear}, which is often unavailable or costly in practice, we develop a novel zero-order, fully distributed algorithm tailored to the convex reformulation of the GNE problem. Each player observes only the values of their local cost and the residuals of their own constraints. Thus, no gradients or Jacobians are available. This information setup has been considered recently in the literature~\cite{tat_kam_TAC,  TatKam2025arxiv, suli}. For games with jointly convex coupling constraints,~\cite{tat_kam_TAC} proposes a payoff-based algorithm whose convergence relies on the existence of a convex potential. The work~\cite{suli} extends this setting to non-potential games with jointly linear constraints via a Tikhonov regularization approach that yields an efficient procedure. The recent paper~\cite{TatKam2025arxiv} leverages the same regularization idea to establish convergence rates for zero-order algorithms in strongly monotone games with shared affine coupling constraints.

The core difficulty in our approach is that the search direction for the local objectives in the reformulated optimization problem is a coupled, global quantity and the objectives' gradients do not coincide with any individual player's local gradients. Our contribution is a few-query scheme that reconstructs the required gradient surrogates from function values alone while keeping bias and variance controlled that preserves fast convergence in a purely distributed implementation. We demonstrate that, under an appropriate choice of the step size parameter, the proposed zero-order procedure possesses a convergence rate of the order $O(1/t)$ which  is (up to dimension-dependent constants)
asymptotically optimal within the class of strongly convex optimization problems~\cite{duchi2015optimal}. 

Thus, our contributions are as follows: 
\begin{itemize}
    \item We present a convex reformulation for GNE problems with general quadratic cost function (without requiring monotonicity) and individual coupling linear equality constraints. We demonstrate that the objective of the latter problem satisfies the Polyak–Łojasiewicz condition, which in particular allows for designing a distributed gradient method with geometrically fast convergence to a GNE in systems endowed with a communication topology; 
    \item We propose a payoff- and residual-only scheme to solve the obtained convex optimization problem that requires no closed-form gradients. Thus, each player uses just local cost evaluations and local constraint residuals to set up her update rules;
    \item With an appropriate step size policy, the zero-order method achieves an 
 $O(1/t)$ rate, which is asymptotically optimal, up to dimension-dependent constants, in the subclass of problems with strongly convex objectives.
\end{itemize}

\subsection{Notations. }The set $\{1,\ldots,n\}$ is denoted by $[n]$. We consider real normed space $\R^D$ and denote by $\R_{+}^D$ its non-negative orthant. The column vector $\xb\in\R^D$ is denoted by $\xb = [x^1,\ldots, x^D]$. We allow vectors to consist of blocks, i.e. $\xb=[\xb^1,\dots,\xb^n]\in\mathbb{R}^{nd}$ is the vector with the blocks $\xb^i\in\mathbb{R}^d$. We
define the projector \(P_i\in\mathbb{R}^{nd\times nd}\) onto block \(j\) by
$P_j \;=\; e_j e_j^\top \,\otimes\, I_d$,
where \(e_j\in\mathbb{R}^n\) is the \(j\)-th standard basis vector and \(I_d\) is the \(d\times d\) identity matrix. Thus, for any \(x=[x^1,\dots,x^n]\in\mathbb{R}^{nd}\),
$P_j x=[0,\dots,0,\;x^j,\;0,\dots,0]\in\R^d$. For any function $f:K\to\R$, $K\subseteq\R^{nd}$, $\nabla_{x^i} f(\xb) = \frac{\partial f(\xb)}{\partial \xb^i}\in\R^d$ is the partial derivative taken in respect to the $\xb^i$th block in the vector argument $\xb\in\R^{nd}$.
 We use $\langle \cdot,\cdot\rangle$ to denote the inner product in $\R^D$.
We use $\|\cdot\|$ to denote the Euclidean norm induced by the standard inner product in $\R^D$.
Given a matrix $H\in\R^{nd\times nd}$, we denote by $H_{(i,:)}\in\R^{d\times nd}$ and $H^{(i,:)}\in\R^{nd\times d}$and the $i$-th block of $d$ rows and columns in $H$ respectively.
$\operatorname{blkdiag}(H_i)_{i=1}^n$ denotes the matrix of the block-diagonal concatenation of the matrices $H_1,\ldots,H_n$, whereas $\operatorname{col}(H_i)_{i=1}^n$ denotes the matrix of the vertical concatenation of the matrices $H_1,\ldots,H_n$. $\sigma_{\min}^+(H)$ denotes the smallest positive singular value of matrix $H$.
We write $a_t = O(b_t)$ (as $t \to \infty$), if there exists a constant $C>0$ 
and an index $t_0$ such that 
$a_t \le C\, b_t$, $\forall t \ge t_0$.
We write $a_t = \Omega(b_t)$ (as $t \to \infty$), if there exists a constant $c>0$ 
 and an index $t_0$ such that 
$a_t \ge c\, b_t$, $\forall t \ge t_0$.
The standard $D$-dimensional Gaussian distribution is denoted by $\mathcal{N}(0,I_D)$ (with zero mean vector and identity covariance matrix).

\section{GNE Problem and Its Equivalent Optimization Formulation}
Let us consider a game $\Gamma(n,\{C_i\},\{J_i\})$ of $n$ players. Each player takes an action\footnote{Here we assume that actions of agents are from  $\R^d$. The analysis extends for the case of individual action dimension $d_i$ for each player $i$.} $\xb^i$ from $\R^d$. We focus on the case where not just the cost functions $J_i(\xb^i,\xb^{-i}): \R^{nd}\to\R$, $i\in[n]$, but also the action sets $C_i$ are coupled through the players' joint action $\xb=[\xb^i,\xb^{-i}]$. We make the following standing assumption regarding the game structure. 
\begin{assumption}\label{assum:standing}
 Each cost function $J_i(\xb^i,\xb^{-i})$ is a quadratic function, namely $J_i(\xb^i,\xb^{-i}) = \frac{1}{2}\xb^TQ_i\xb + \rb_i^T\xb + k_i$, where $Q_i\in\R^{nd\times nd}$ is symmetric, $\rb_i = [\rb_i^1,\ldots,\rb_i^n]\in\R^{nd}$ ($\rb_i^j\in\R^d$ for all $j\in[n]$), $k_i\in\R$, such that $J_i$ is convex in $\xb^i\in\R^d$ for any fixed $\xb^{-i}\in\R^{(n-1)d}$. Each coupled action set $C_i(\xb^{-i})$ is represented by a system of linear equations, i.e.  $C_i(\xb^{-i}) = \{\xb^i\in\R^d: \, A_i\xb-\bb_i = \boldsymbol{0}\}$ with some $A_i\in\R^{m_i\times nd}$ and $\bb_i\in\R^{m_i}$. 
\end{assumption}

Given assumption above, we face a game with coupling constraints. Thus, a solution concept is defined by a \emph{generalized Nash equilibrium} (GNE) that is a joint action from which no player has any incentive to unilaterally deviate.
\begin{definition}\label{def:GNE}
 A joint action $\xb^*\in  C = \cap_i{C_i}$ in the game $\Gamma=\Gamma(n,\{C_i\},\{J_i\})$ above is called a \emph{generalized Nash equilibrium} (GNE) if $J_i(\xb^{*,i},\xb^{*,-i})\le J_i(\xb^i,\xb^{*,-i})$ for all $i\in[n]$ and $\xb^i\in C_i(\xb^{*,-i})$.
 \end{definition}
 We assume there exists a GNE in $\Gamma$. 
 \begin{assumption}\label{assum:exist}
  The game $\Gamma$ possesses a GNE solution.      
 \end{assumption}
 Note that a GNE exists, in particular, if the action sets are convex and the pseudo-gradient of the game, namely the mapping $\Mb(\xb) = [\nabla_{\xb^1}J_1(\xb), \ldots, \nabla_{\xb^n}J_n(\xb)]$, is strongly monotone~\cite{facchinei2007generalized}\footnote{We recall that a mapping $\Mb:\R^D\to\R^D$ is called strongly monotone over $\R^D$, if $\langle\Mb(\xb)-\Mb(\yb),\xb-\yb\rangle\ge \mu\|\xb-\yb\|^2$ for some $\mu>0$ (if $\mu = 0$, the mapping $\Mb$ is called merely monotone).}. However, we do not assume monotonicity of the game. We merely require a quadratic structure of the cost functions (see Assumption~\ref{assum:standing}). 
 
We emphasize that the players in the game $\Gamma$ \textit{do not share} their coupling constraints but rather possess individual ones $C_i(\xb^{-i})$ for each $i\in[n]$. Thus, the concept of a so-called variational GNE, which is a solution to a specific variational inequality, and the corresponding computation approaches to find a variational GNE \cite{LinEqGram, TatKam2025arxiv} cannot be applied to $\Gamma$. 
To be able to deal with the coupling constraints, we use the well-known results on GNE presented in Theorem 8 of~\cite{facchinei2007generalized}, by means of a KKT-type condition. 
\begin{prop}\label{prop:1}
 Given Assumption~\ref{assum:standing}, a joint action $\xb^*$ is a GNE of the game $\Gamma$ if and only if there exist $\blab^{*,i}\in\R^{m_i}$, $i\in[n]$, such that for all $i\in[n]$ the following holds: 
 \begin{align}\label{eq:GNE}
     &\nabla_{\xb^i} J_i(\xb^*) + A_{i(i,:)}^T\blab^{*,i} = \boldsymbol{0},\cr
     &A_i\xb^*-\bb_i = \boldsymbol{0}.
 \end{align}
\end{prop}
We call the vectors $\blab^{*,i}$ the dual variables for the solution $\xb^*$. If Assumption~\ref{assum:exist} holds, then there exists a primal-dual pair $\xb^*$ and $\blab^*$ in $\Gamma$ satisfying~\eqref{eq:GNE} which implies that $\xb^*$ is a GNE in $\Gamma$.

Taking into account the result in Proposition~\ref{prop:1} and the form of the cost functions (see again Assumption~\ref{assum:standing}) resulting in the relation $\nabla_{\xb^i} J_i(\xb) = H_i \xb + \rb_i^{i}\in\mathbb{R}^d$ with $H_i=\tfrac12\big[Q_i+Q_i^\top\big]_{(i,:)}\in\mathbb{R}^{d\times nd}$, we conclude that a GNE $\xb^*$ and its corresponding dual variable $\blab^*=[\blab^{*,1},\ldots, \blab^{*,n}] \in \R^{m}$, where $m=\sum_{i=1}^n m_i$, can be found as a solution to the following optimization problem: 
\begin{align}\label{eq:DOP}
  \min_{\xb,\blab} F(\xb,\blab),
\end{align} 
where $F(\xb,\blab) = \sum_{i=1}^n f_i(\xb,\blab^i)$ with
\[f_i(\xb,\blab^i)=\| H_i \xb + \rb_i^{i} + A_{i(i,:)}^T\blab^i \|^2 + \|A_i\xb-\bb_i\|^2.\] 

At this point we note that the function $F(\xb,\blab)$ is a \emph{gap function}, as $F(\xb,\blab)\ge 0$. Therefore, $\xb^*$ is a GNE if and only if  $F(\xb^*,\blab^*)=0$ for some $\lambda^*$. In the following, our goal will be to minimize the gap function $F$. 

Let us introduce the notation $\zb = [\xb,\blab]$. 
It is straightforward to see that the function $F$ in the problem~\eqref{eq:DOP} is of the form 
\begin{align}\label{eq:F}
    F(\zb) = \|G\zb+\eb\bigr\|^2.
\end{align}
 In above, 
 \[
G=\begin{bmatrix}
H_{\mathrm{stk}} & A^{\top}_{\mathrm{blk}}\\[2mm]
A_{\mathrm{stk}} & 0
\end{bmatrix},
\qquad
\eb=\begin{bmatrix}
\rb_{\mathrm{stk}}\\[1mm]
-\bb_{\mathrm{stk}}
\end{bmatrix},
\]
with $
H_{\mathrm{stk}}=\operatorname{col}(H_i)_{i=1}^n$, $A_{\mathrm{stk}}=\operatorname{col}(A_i)_{i=1}^n$,
$A^{\top}_{\mathrm{blk}}=\operatorname{blkdiag}(A_{i(i,:)}^{\top})_{i=1}^n$, 
$\rb_{\mathrm{stk}}=[\rb_1^{1},\dots,\rb_n^{n}]$, $\bb_{\mathrm{stk}}=[\bb_1,\dots,\bb_n]$.
Because F is of this form, we conclude that the function $F$ satisfies a so called Polyak--Łojasiewicz condition (PL condition), i.e.\footnote{Here we take into account that the minimum of $F$ is attained (see Assumption~\ref{assum:exist}) and $F^*=0$.}
\begin{align}\label{eq:PL}
    \|\nabla F(\zb)\|^2 \ge 2\mu_F F(\zb),\quad \forall \bx,\blab,
\end{align}
where 
\begin{align}\label{eq:mu}
\mu_F = 2\,\big(\sigma_{\min}^+(G)\big)^2.
\end{align}
\begin{rem}\label{rem:1oa}The Polyak--Łojasiewicz condition has recently gained a lot of attention in the optimization literature. In particular, it has been shown that the gradient descent algorithm converges linearly to an unconstrained optimal value of any Lipschitz smooth function satisfying  the PL condition~\cite{PL}. Due to its quadratic structure, the function $F$ possesses Lipschitz continuous gradients. Hence, one can directly apply the result from the paper~\cite{PL_distr} to set up a distributed procedure to solve~\eqref{eq:DOP} provided a definite communication structure in the system and access to the first order information meaning that each player can calculate $\nabla f_i$ at any given point. The corresponding algorithm is proven to converge exponentially fast to a solution $[\xb^*,\blab^*]$, given an appropriately chosen constant step size.
\end{rem}

In this work, we relax the assumption regarding the access to the gradient information and propose a distributed zero-order approach to solve the problem~\eqref{eq:DOP}, i.e., to find a GNE in the game $\Gamma$.

\section{Zero-order Distributed Approaches}
We assume each player $i$ in the game $\Gamma$ has access to her \emph{zero-order information}. It means that given a time step $t$ and a query point $\tilde{\xb}(t)$, the player $i$ observes the local cost value $J_i(\tilde{\xb}(t))$ and the violation of the local constraints, namely $A_i\tilde{\xb}(t)-\bb_i$.
Note that using this information and controlling her local dual variable $\blab^i(t)$, each player $i$  can determine the value of her local Lagrangian at any $\tilde{\xb}(t)$ and $\blab^i(t)$:
 \begin{align}\label{eq:Lagr}
 L_i(\tilde{\xb}(t),\blab^i(t)) =  J_i(\tilde{\xb}(t)) + \langle\blab^i(t), A_i\tilde{\xb}(t)-\bb_i\rangle.
 \end{align}

To solve the problem~\eqref{eq:DOP} in a distributed way, having access merely to the local zero-order information described above, each player constructs an estimation of the partial derivatives $\frac{\partial F}{\partial \xb^i}$ and $\frac{\partial F}{\partial \blab^i}$ of the objective function $F$. Based on these estimations, each agent performs a gradient descent step. Then, the overall procedure mimics the gradient descent algorithm applied to the function $F$.
And since the function $F$ satisfies PL-condition, we expect such a procedure to approach a solution with a fast rate. 

To discuss appropriate gradient estimations, let us first consider two parts of each local function $f_i(\xb,\blab^i) = h_i(\xb,\blab^i) + c_i(\xb)$, where 
\begin{align}\label{eq:2p}
h_i(\xb,\blab^i) &= \| H_i \xb + \rb_i^{i} + A_{i(i,:)}^T\blab_i \|^2 = \|\nabla_{\xb^i}L_i(\xb,\blab^i)\|^2, \cr
c_i(\xb) &= \|A_i\xb-\bb_i\|^2,
\end{align}
such that 
\begin{align}\label{eq:sum}
F(\xb,\blab) = \sum_{i=1}^n h_i(\xb,\blab^i) +  \sum_{i=1}^n c_i(\xb).
\end{align}
As we assume access to the violation of the local constraints, the value of $c_i$ can be directly calculated for any query point $\tilde{\xb}(t)$. In this way, many known techniques can be applied to let the player $i$ estimate the corresponding partial derivatives of $c_i$ at a definite point~\cite{NesterovSpokoiny, tat_kam_TAC}. As for the gradients of $h_i$, we cannot directly apply the same approaches. Indeed, here we need to use the value observations of  the Lagrangian $L_i(\xb,\blab^i)$ to estimate the derivatives of the function $h_i(\xb,\blab^i) = \|\nabla_{\xb^i}L_i(\xb,\blab^i)\|^2$, not the derivatives of the Lagrangian $L_i(\xb,\blab^i)$ itself. In the following discussion we present appropriate estimations of $\nabla_{\xb^i}h_i(\xb,\blab^i)$ as well as of $\nabla_{\blab^i}h_i(\xb,\blab^i)$ using in particular a quadratic structure of $h_i$. Finally, we notice that due to the proposed reformulation of the initial GNE problem in terms of the single optimization problem~\eqref{eq:DOP} with the objective function in~\eqref{eq:sum}, we need to assume existence of an \emph{aggregator} in the system which helps to turn everyone’s local zero-order measurements into the final estimate of the sum of the corresponding partial derivatives.

In the further analysis we will need the following straightforward calculations of the gradients:
\begin{align}\label{eq:nab_h}
    \nabla_{\xb^i}h_j(\xb,\blab^j) = 2 (H_j^{(:,i)})^{\top} \nabla_{\xb^i}L_j(\xb,\blab^j).
\end{align}
 
\subsection{Gradient Estimations}
 To estimate the gradients $\frac{\partial F}{\partial \xb^i}$ and $\frac{\partial F}{\partial \blab^i}$, $i\in[n]$, at the current joint action $\xb(t)$, 
each player $i$ generates independent Gaussian directions $\xi_x^i(t)\sim\mathcal N(0,I_{d})$, $\xi_{\lambda}^i(t)\sim\mathcal N(0,I_{m_i})$, $\eta^i(t)\sim\mathcal N(0,I_{d})$. Thus, the following four query points can be obtained: 
\begin{align*}
   \tilde{\xb}_1(t) &= \xb(t)  - \sigma \eta(t), \quad \tilde{\xb}_3(t) = \xb(t) + \delta \xi_x(t) - \sigma \eta(t),\cr
   \tilde{\xb}_2(t) &= \xb(t)  + \sigma \eta(t),\quad   \tilde{\xb}_4(t) = \xb(t) + \delta \xi_x(t) + \sigma \eta(t),
 \end{align*}
 where $\xi_x(t)  = [\xi^{1}_{x}(t), \ldots, \xi^n_{x}(t)]\sim\mathcal N(0,I_{nd})$, $\xi_{\lambda}(t) = [\xi^1_{\lambda}(t),\ldots,\xi^n_{\lambda}(t)]\sim\mathcal N(0,I_{m})$, $\eta(t) = [\eta^1(t), \ldots, \eta^n(t)]\sim\mathcal N(0,I_{nd})$, and  $\sigma>0$, $\delta>0$.
Then each player $i$ samples the value of her local Lagrangian $L_i$ at the following four points:
\begin{align}\label{eq:4p}
   \yb^i_1(t) &= [\tilde{\xb}_1(t),\; \blab^i(t) ], \quad \yb^i_3(t) = [\tilde{\xb}_3(t),\; \blab^i(t) + \delta  \xi^i_{\lambda}(t)],\cr
   \yb^i_2(t)& = [\tilde{\xb}_2(t),\; \blab^i(t) ], \quad \yb^i_4(t) = [\tilde{\xb}_4(t),\; \blab^i(t) + \delta  \xi^i_{\lambda}(t)].
\end{align}
 Next, the player $i$ constructs the following two–point differences for her Lagrangian function $L_i$ and local function $c_i$ with an intend to approximate their full gradients:
 \begin{align}\label{eq:Delta}
\Delta_1^{(i)}(t)&=\frac{L_i(\yb^i_2(t))-L_i(\yb^i_1(t))}{2\sigma},\cr
\Delta_2^{(i)}(t)&=\frac{L_i(\yb^i_4(t))-L_i(\yb^i_3(t))}{2\sigma},\cr
\Delta_3^{(i)}(t)&=\frac{c_i(\tilde{\xb}_2(t))-c_i(\tilde{\xb}_1(t))}{2\sigma}.
\end{align}
Based on the constructions above, the player $i$ defines the following squared directional increments:
\begin{align}\label{eq:S}
S_1^{(i)}(t)&=\frac{(\Delta_2^{(i)}(t))^2-(\Delta_1^{(i)}(t))^2}{\delta},\cr
\quad
S_2^{(i)}(t)&=S_1^{(i)}(t)\,\|\eta^i(t)\|^2.
\end{align}

Next, we introduce the \emph{aggregator} of the system. Its goal is to aggregate the values $\tfrac{1}{2}(S_2^{(i)}(t)-d\,S_1^{(i)}(t))$ as well as $\Delta_3^{(i)}(t)$, which are locally constructed by each player $i$ by means of the available zero-order information, and to broadcast back to each player the sums 
\begin{align}\label{eq:SD}
&S(t) =\tfrac{1}{2}\sum_{j=1}^n(S_2^{(j)}(t)-d\,S_1^{(j)}(t)),\cr
&D(t) = \sum_{j=1}^n\Delta_3^{(j)}(t).
\end{align}
After receiving $S(t)$ and $D(t)$, each player constructs her estimation of $\frac{\partial F}{\partial \xb^i}$ and $\frac{\partial F}{\partial \blab^i}$ at the point $[\xb(t),\blab(t)]$ respectively as follows: 
\begin{align}\label{eq:grEst}
    \zeta_{\xb^i}(t)\;&=\;S(t)\;\xi_x^i(t)+D(t)\eta^i(t)\in\R^d,\cr
    \zeta_{\blab^i}(t)\;&=\;\tfrac12\,\big(S_2^{(i)}-d\,S_1^{(i)}\big)\,\xi_{\lambda}^i(t)\in\R^{m_i}.
\end{align}
We note that the aggregated value $S(t)$ is required merely for the estimation of $\frac{\partial F}{\partial \xb^i}$. The estimation of $\frac{\partial F}{\partial \blab^i}$ can be constructed locally by each player $i$ based on $S_1^{(i)}(t)$ and $S_2^{(i)}(t)$. The reason is that each local function $h_i(\xb,\blab^i)$ in the optimization problem under consideration does not depend on $\blab^{-i}$. 

To be able to use the gradient estimations constructed above in the gradient-based procedure, we need to figure out their stochastic properties, namely the mean and the second moment values. We denote a sigma-algebra generated by the random variables $\{\xi_{x}(k),\xi_{\lambda}(k),\eta(k)\}_{k\in[t]}$ by $\EuScript F_t$. The estimation properties are formulated in the following lemma. 
\begin{lem}\label{lem:grEst}
  The  estimations in~\eqref{eq:grEst} are unbiased gradient estimations of the function $F$ at the point $[\xb(t), \blab(t)]$, i.e. for any $\delta>0$ and $\sigma>0$
  \begin{align*}
      \E\{\zeta_{\xb^i}(t)\;|\;\EuScript F_t\} &= \nabla_{\xb^i}F(\xb(t),\blab(t)),\cr
      \E\{\zeta_{\blab^i}(t)\;|\;\EuScript F_t\} &= \nabla_{\blab^i}F(\xb(t),\blab(t)).
  \end{align*}
Moreover, 
\begin{align*}
\E\{\big\|\zeta_{\xb^i}(t)\big\|^2\;|\;\EuScript F_t\}&= g_1(\xb(t),\blab(t))\;+\;O(\delta^2),\cr
\E\{\big\|\zeta_{\blab^i}(t)\big\|^2\;|\;\EuScript F_t\}&=g_2(\xb(t),\blab(t))\;+\;O(\delta^2),
\end{align*}
where $g_1$ and $g_2$ are some quadratic functions of the variables $\xb(t)$ and $\blab(t)$. 
\end{lem}
\begin{proof}
    First, we demonstrates that $\zeta_{\xb^i}(t)$ and $\zeta_{\blab^i}(t)$
 are unbiased estimations of $\frac{\partial F(\xb(t))}{\partial \xb^i}$ and $\frac{\partial F(\xb(t))}{\partial \blab^i}$ respectively. 
 As each Lagrangian $L_i(\xb,\blab^i)$ is a quadratic function in $\xb$, we obtain from~\eqref{eq:Delta} 
 \begin{align}\label{eq:Delta1}
     &\Delta_1^{(i)}(t)\cr
     &=\frac{L_i(\xb(t)  + \sigma \eta(t),\; \blab(t))-L_i(\xb(t)  - \sigma \eta(t),\; \blab(t))}{2\sigma}\cr
     &=\langle\nabla_{\xb}L_i(\xb(t),\blab^i(t)), \;\eta(t)\rangle,\cr
     &\Delta_2^{(i)}(t)=\langle\nabla_{\xb}L_i(\xb(t)+ \delta \xi_x(t),\;\blab^i(t)+ \delta  \xi^i_{\lambda}(t)),\; \eta(t)\rangle,\cr
     &\Delta_3^{(i)}(t)=\langle\nabla_{\xb}c_i(\xb(t)), \;\eta(t)\rangle
 \end{align}
 Since $\nabla_{\xb}L_i(\xb(t),\blab^i(t))= Q_i \xb(t) + \rb_i + A_i^T\blab^i(t)$ is an affine function in $[\xb(t),\blab^i(t)]$, we conclude that
 \begin{align}
     &\nabla_{\xb}L_i(\xb(t)+ \delta \xi_x(t),\;\blab^i(t)+ \delta  \xi^i_{\lambda}(t)) \cr
     &\quad=\nabla_{\xb}L_i(\xb(t),\blab^i(t)) + \delta\big(Q_i \xi_x(t) + A_i^\top \xi^i_\lambda(t)\big).
 \end{align}
 Hence,
 \begin{align}\label{eq:Delta2}
     \Delta_2^{(i)}(t)=&\langle\nabla_{\xb}L_i(\xb(t),\;\blab^i(t)), \eta(t)\rangle\cr
     &+ \delta \langle Q_i \xi_x(t) + A_i^\top \xi^i_\lambda(t),\; \eta(t)\rangle\cr
     =&\Delta_1^{(i)}(t) + \delta \langle Q_i \xi_x(t) + A_i^\top \xi^i_\lambda(t),\; \eta(t)\rangle.
 \end{align}
 For the sake of notation simplicity, let us introduce the following shorthand (omitting $t$) in the corresponding notations: 
 \begin{align*}
     &\nb_i = \nabla_{\xb}L_i(\xb,\blab), \quad \pb_i = Q_i \xi_x + A_i^\top \xi^i_\lambda, \quad \qb_i = \nabla_{\xb}c_i(\xb). 
 \end{align*}
  Thus, 
 \begin{align}\label{eq:S1}
&S_1^{(i)}= \frac{(\Delta_2^{(i)})^2-(\Delta_1^{(i)})^2}{\delta}=2\langle\nb_i,\eta\rangle\langle\pb_i,\eta\rangle+\delta\,\langle\pb_i(t),\eta\rangle^2.
\end{align}
Next, according to the definition of $S_2^{(i)}$ (see~\eqref{eq:S}), we obtain
\begin{align}
S_2^{(i)}-dS_1^{(i)}=(2\langle\nb_i,\eta\rangle\langle\pb_i,\eta\rangle+\delta\,\langle\pb_i,\eta\rangle^2)(\|\eta^i\|^2-d)        
\end{align}
On the other hand, 
$\Delta_3^{(i)}=\langle\qb_i, \eta\rangle$, and, hence, we conclude that (see~\eqref{eq:SD} and~\eqref{eq:grEst})
\begin{align}\label{eq:zeta}
    &\zeta_{\xb^i}=\sum_{j=1}^n\langle\nb_j,\eta\rangle\langle\pb_j,\eta\rangle(\|\eta^j\|^2-d) \xi_x^i\cr
    &\,+\frac{1}{2}\sum_{j=1}^n\delta\,\langle\pb_j,\eta\rangle^2(\|\eta^j\|^2-d) \xi_x^i+\sum_{j=1}^n\langle\qb_j, \eta\rangle\eta^i.
\end{align}
We proceed with estimation of the conditional expectation of $\zeta_{\xb^i}$ in respect to the Gaussian variables $\eta$, $\xi_{x}$, and $\xi_{\lambda}$. To do so, we use the following property of the Gaussian distribution: For independent $u_j\sim\mathcal N(0,I_{d})$, $j\in[n]$, $u=[u_1,\ldots,u_n]$ and any $a,b\in\mathbb{R}^{nd}$,
$
\mathbb{E}\big[\langle a, u\rangle \langle b, u\rangle\,\|u_j\|^2\big]
= \langle a,(d I_{nd}+2P_j)\,b\rangle$ (see \textbf{Notations} for definition of the operator $P_j$).
Applying the last relation, we calculate\footnote{More details regarding~\eqref{eq:zeta1} can be found in Appendix~\ref{app:details}.} $\E\{\zeta_{\xb^i}\;|\;\eta\}$: 
\begin{align}\label{eq:zeta1}
\E\{\zeta_{\xb^i}\;&|\;\eta\} = 2\sum_{j=1}^n\langle\nb_j,P_j\pb_j\rangle\xi_x^i\cr
&+\mbox{Pol}_2(\xi_{x},\xi_{\lambda})\xi_x^i + \nabla_{\xb^i}\sum_{j=1}^nc_j(\xb),
\end{align}
where $\mbox{Pol}_2(\xi_{x},\xi_{\lambda})$ is a second order polynomial dependent on $\xi_{x}$ and $\xi_{\lambda}$. Finally, taking into account that $\nb_j = \nabla_{\xb}L_j(\xb,\blab)$, $\pb_j = H_j\xi_x + A_j^\top \xi_\lambda$ for all $j\in[n]$, and $\E\{\mbox{Pol}_2(\xi_{x},\xi_{\lambda})\xi_x^i \;|\;\xi_x\}= 0$, we get
\begin{align}
&\E\{\zeta_{\xb^i}\;|\;\mathcal F_t\} = 2\sum_{j=1}^n\E\{\langle\nb_j,P_j\pb_j\rangle\xi_x^i|\; \xi_{x}, \xi_{\lambda}\}\cr
&\qquad\qquad+ \nabla_{\xb^i}\sum_{j=1}^nc_j(\xb)\cr
&=2\sum_{j=1}^n(H_j^{(:,i)})^{\top} \nabla_{\xb^i}L_j(\xb,\blab^j)+ \nabla_{\xb^i}\sum_{j=1}^nc_j(\xb),
\end{align}
which implies, due to the equality~\eqref{eq:nab_h}, that
\begin{align*}
    \E\{\zeta_{\xb^i}(t)\;|\;\EuScript F_t\} &= \nabla_{\xb^i}\sum_{j=1}^n[h_j(\xb(t),\blab^i(t)) + c_j(\xb(t))]\cr
    &=\nabla_{\xb^i}F(\xb(t),\blab(t)).
\end{align*}
Analogous calculations demonstrate that $\E\{\zeta_{\blab^i}(t)\;|\;\EuScript F_t\} = \nabla_{\blab^i}F(\xb(t),\blab(t))$. 
As for the second moments, it suffices to use the expression~\eqref{eq:zeta}  and follow the order of the Gaussian vectors $\xi_x$, $\xi_{\lambda}$, and $\eta$ therein while taking the conditional expectation of $\|\zeta_{\xb^i}\|^2$. Again, similar direct calculation can be applied to get the relation for  the conditional expectation of $\|\zeta_{\blab^i}\|^2$.
\end{proof}

\subsection{Distributed Algorithm and Main Result}
Having access to the estimations of the partial derivatives $\frac{\partial F}{\partial \xb^i}$ and $\frac{\partial F}{\partial \blab^i}$ constructed in~\eqref{eq:grEst}, each player $i$ updates their current action and dual variable according to the gradient descent iterates: 
\begin{align}\label{eq:grDes}
    \xb^i(t+1) &= \xb^i(t) - \gamma_t\;\zeta_{\xb^i}(t)\cr
    \blab^i(t+1) &= \blab^i(t) - \gamma_t\;\zeta_{\blab^i}(t),
\end{align}
where $\gamma_t$ is a step size to be set up. The overall distributed zero-order procedure can be now summarized in Algorithm~\ref{alg:1}. 

\begin{algorithm}[H]\caption{Zero-order Procedure for Player $i$}\label{alg:1}
{\bf Input:} Each player process their local action $\xb^i$ and dual variable $\blab^i$ with initial values $\xb^i(0)$, $\blab^i(0)$ based on the access to the local cost value $J_i(\tilde{\xb})$ and the violation of the local constraints $A_i\tilde{\xb}-\bb_i$ at any query point $\tilde{\xb}$, $i\in [n]$; additionally consider a step size sequence $\{\gamma_t>0\}$ as well as variance parameter sequences $\{\sigma>0\}$, $\{\delta>0\}$.  

\noindent
{\bf For} $t=1,2,\ldots$, player $i$ performs the following steps:

\begin{itemize}
  \item Sample $\xi_x^i(t),\,\xi_\lambda^i(t),\,\eta^i(t), \; \forall i$, as independent standard Gaussian vectors and construct 4 query points $\yb_k(t))$, $k=1,2,3,4$, via~\eqref{eq:4p}.

  \item Evaluate $L_i(\yb_k(t)), \; \forall i \in [n]$, $k=1,2,3,4$ as well as $c_i(\yb_1(t))$, $c_i(\yb_2(t)), \; \forall i \in [n]$ and compute $\Delta_1^{(i)},\,\Delta_2^{(i)},\,\Delta_3^{(i)}, \; \forall i \in [n]$ via \eqref{eq:Delta}.
  
  \item Compute $S_1^{(i)},\,S_2^{(i)}, \; \forall i \in [n]$ via \eqref{eq:S}.
  
  \item Send $\big(\tfrac12(S_2^{(i)}-dS_1^{(i)}),\,\Delta_3^{(i)}\big)$ to aggregator; receive $S(t),D(t)$.
  
  \item Set $\zeta_{\xb^i}(t)=S(t)\xi_x^i(t)+D(t)\eta^i(t)$ and $\zeta_{\blab^i}(t)=\tfrac12\big(S_2^{(i)}{-}dS_1^{(i)}\big)\xi_{\lambda}^i(t)$.
  
  \item Update $\xb^i(t{+}1)=\xb^i(t)-\gamma_t\zeta_{\xb^i}(t)$ and $\blab^i(t{+}1)=\blab^i(t)-\gamma_t\zeta_{\blab^i}(t)$.
\end{itemize}
\end{algorithm}

\begin{rem}
    As we can see, Algorithm~\ref{alg:1} follows the logic of the zero-order approach: Each player uses merely actual values of her cost function and constraint residual to construct the estimates $S_1^{(i)}$ and $S_2^{(i)}$ which are sent to the aggregator. We also notice that the procedure requires four query points $\yb_k(t)$, $k=1,2,3,4$, to estimate the corresponding gradients at each $t$. It uses  $\yb_1(t)$ and $\yb_2(t)$ to be able to approximate $\nabla_{\xb^i} L_i(\xb,\blab^i)$ and $\nabla_{\xb} c_i(\xb)$, whereas the next pair of query points, i.e. $\yb_3(t)$ and $\yb_4(t)$, are used to estimate $\nabla_{\xb}\|\nabla_{\xb^i} L_i(\xb,\blab^i)\|^2$. We apply two-point queries for each gradient approximation above to be able to achieve an unbiased estimation with an appropriate upper bound for the second moment (see, for example~\cite{TatKam2025arxiv}, for more discussion on the power of two-point estimations).   
    As for the practical aspect, the presented four-point estimation assumes the players to deviate at each $t$ from their current joint action four times as in~\eqref{eq:4p} and query or play the corresponding joint actions in a synchronous way to experience the cost values $J_i$, $i\in[n]$, and the violation of their constraints at these points. 
\end{rem}
Next, we formulate the main convergence result for Algorithm~\ref{alg:1}. 
\begin{theorem}\label{th:main}
 Let Assumptions~\ref{assum:standing},~\ref{assum:exist} hold for the game $\Gamma$. Then, under the choice $\gamma_t = \frac{g}{t}$ with $g > \frac{1}{\mu_F}$ and given $\sigma>0$, $\delta>0$, the iterates $\xb(t) = [\xb^1(t),\ldots,\xb^n(t)]$ converge almost surely to some (random) GNE. Moreover, the expectation of the gap function, i.e. $\E F_t = \E F(\xb(t),\blab(t))$, converges to $0$ with the rate $O\left(\frac{1}{t}\right)$, namely $\E F_t = O\left(\frac{1}{t}\right)$ as $t\to\infty$.
\end{theorem}
\begin{proof}
 We start by noticing that the procedure~\eqref{eq:grDes} can be considered an unbiased stochastic gradient descent in respect to the gap function $F(\xb,\blab)$: 
 \begin{align}\label{eq:z}
     \zb(t+1) = \zb(t) - \gamma_t\zeta(t).
 \end{align}
 Here, as before, $\zb(t) = [\xb(t),\blab(t)]$ and $\zeta(t) = [\zeta_{\xb^1}(t), \ldots, \zeta_{\xb^n}(t),\zeta_{\blab^1}(t), \ldots, \zeta_{\blab^n}(t)]$ denotes the constructed estimation of $\nabla F$ at the point $\zb(t)$. Using the notations above, the relation~\eqref{eq:z}, as well as the quadratic structure of the function $F(\zb)$, in particular its Lipschitz smoothness with some constant $L_F>0$, we obtain
 \begin{align}\label{eq:Ineq1}
 &F(\zb(t+1)) \le F(\zb(t)) + \langle\nabla F(\zb(t)), \; \zb(t+1)-\zb(t)\rangle \cr
 &\qquad\qquad\qquad\qquad + \frac{L_F}{2}\|\zb(t+1)-\zb(t)\|^2 \cr
 &= F(\zb(t)) - \gamma_t \langle\nabla F(\zb(t)), \; \zeta(t)\rangle + \frac{L_F\gamma_t^2}{2}\|\zeta(t)\|^2. 
 \end{align}
 Next, we notice that Lemma~\ref{lem:grEst} and the quadratic structure of $F$ imply that, for some $M>0$,
 \begin{align}\label{eq:lem}
     &\E_t\{\zeta(t)\} = \nabla F(\zb(t)),\quad
     \E_t\{\|\zeta(t)\} \le M( F(\zb(t)) + \delta^2),
 \end{align} 
 where we use the notation $\E_t\{\cdot\} = \E\{\cdot\;|\;\EuScript F_t\}$
 Thus, taking the conditional expectation in the inequality~\eqref{eq:Ineq1}, we conclude that  for all $t\ge \frac{gML_F}{2\mu_F}$ almost surely
 \begin{align}\label{eq:RS}
 &\E_t\{F(\zb(t+1))\} \le F(\zb(t)) - \gamma_t \|\nabla F(\zb(t))\|^2\cr
 &\qquad\qquad\qquad\quad + \frac{ML_F\gamma_t^2}{2}( F(\zb(t)) + \delta^2)\cr
 &=\left(1 - 2\mu_F\gamma_t + \frac{ML_F\gamma_t^2}{2}\right) F(\zb(t)) + \frac{ML_F\delta^2\gamma_t^2}{2}\cr
 &\le \left(1 - \frac{\mu_F g}{t} \right) F(\zb(t)) + \frac{ML_F\delta^2g^2}{2t^2},
 \end{align}
 where the last inequality holds due to the choice of $\gamma_t = \frac{g}{t}$ and $t\ge \frac{gML_F}{2\mu_F}$. Thus, we can apply the result of Robbins and Siegmund (see Appendix~\ref{app:RS}) to~\eqref{eq:RS} and conclude that almost surely $F(\zb(t))$ converges to 0. Moreover, since $F(\zb)\to\infty$ as $\|\zb\|\to\infty$, it implies that $\|\zb(t)\|$ stays uniformly upper bounded with probability 1.  
 Furthermore, taking the full expectation in~\eqref{eq:RS}, we obtain 
 \begin{align}\label{eq:ERS}
 &\E F(\zb(t+1))\le  \left(1 - \frac{\mu_F g}{t} \right) \E F(\zb(t)) + \frac{ML_F\delta^2g^2}{2t^2},
 \end{align}
 implying that $\lim_{t\to\infty} \E F(\zb(t)) = O\left(\frac{1}{t}\right)$  (due to the choice of $g>\frac{1}{\mu_F}$ and the Chung's lemma formulated in Appendix~\ref{app:Chung}). 

 Now we turn our attention to evolvement of the iterates $\zb(t)$. For this purpose we use the expression~\eqref{eq:F} for the gap function $F$ and denote its gradient as $\nabla F(\zb) = H\zb + \lb$, where $H =2G^\top G$, $\lb =2G^\top\eb$.  
 Let $P_{rH}[\cdot]$ and $P_{kH}[\cdot]$ be the operators of orthogonal projection onto 
$\mathrm{range}(H)$ and $\ker(H)$, respectively, and 
$\zb^r(t) = P_{rH}[\zb(t)-\zb^*]$, $\zb^k(t) = P_{kH} [\zb(t)-\zb^*]$, where $\zb^* = [\xb^*, \blab^*]$ is some solution to~\eqref{eq:DOP}, which implies that $\xb^*$ is a GNE in $\Gamma$ (see Proposition~\ref{prop:1}). 
Then, due to linearity of orthogonal projection operator onto a linear subspace and the fact that $P_{rH}H[\cdot] = HP_{rH}[\cdot]$, $P_{kH}H[\cdot] = \boldsymbol{0}$,  we obtain from~\eqref{eq:z} that
\begin{align*}
\zb^r(t+1) &= \zb^r(t) - \gamma_t\big(H \zb^r(t) + P_{rH}[\zeta(t)]\big),\\
\zb^k(t+1) &= \zb^k(t) - \gamma_t\,P_{kH}[\zeta(t)].
\end{align*}
Thus, 
\begin{align*}
   \|\zb^r(t+1)\|^2 = \| \zb^r(t)\|^2 &-2\gamma_t\langle H \zb^r(t) + P_{rH}[\zeta(t)], \zb^r(t)\rangle\cr
   &+\gamma_t^2\|H \zb^r(t) + P_{rH}[\zeta(t)]\|^2.
\end{align*}
Using again~\eqref{eq:lem} and the fact that $\|\zb(t)\|$ is almost surely uniformly bounded, we conclude that
\begin{align*}
   \E_t\|\zb^r(t+1)\|^2 &= \| \zb^r(t)\|^2 -2\gamma_t\langle H \zb^r(t) , \zb^r(t)\rangle+O(\gamma_t^2)\cr
   &\le (1-2\lambda_H\gamma_t)\| \zb^r(t)\|^2 + O(\gamma_t^2),
\end{align*}
where  $\lambda_H>0$ is a minimum positive eigenvalue of $H$. In the last inequality we use the fact that the matrix $H$ is positive definite on $\mathrm{range}(H)$ and $\zb^r(t)\in \mathrm{range}(H)$. Thus, according to Lemma~\ref{lem:RS} in Appendix~\ref{app:RS}, almost surely 
$\lim_{t\to\infty}\zb^r(t) =  0$.
As for the component $\zb^k(t+1)$, 
\[
\zb^k(t+1)-\zb^k(t) = -\gamma_t P_{kH}\zeta(t),
\]
where $\{P_{kH}\zeta(t)\}$ is a martingale difference such that
$\E\{\|P_{kH}\zeta(t)\|^2 \mid \mathcal F_t\}$ is almost surely uniformly upper bounded (again due to upper bounded $\|\zb(t)\|$ almost surely).
Since $\sum_t \gamma_t^2 < \infty$, 
the series $\sum_t \gamma_t \eta_t$ converges almost surely 
(see the martingale 
convergence theorem in Appendix~\ref{app:mart}). 
Therefore $\{\zb^k(t)\}$ is a Cauchy sequence almost surely. Thus, there exists (random) $\zb^k_\infty \in \ker(H)$ such that almost surely
$\lim_{t\to\infty}\zb^k(t) =  \zb^k_\infty$.
Hence, 
$\zb(t)-\zb^*= \zb^r(t) + \zb^k(t) \to \zb^k_\infty$ almost surely as $t\to\infty$. 
It implies that almost surely 
$\zb(t)\to\zb_\infty = \zb^\star + \zb^k_\infty$ as $t\to\infty$.
Since $\zb^k_\infty \in \ker(H)=\ker(G)$, we have that $F(\zb_\infty) = \|G \zb_\infty + \eb\|^2 = \|G \zb^\star + G\zb^k_\infty + \eb\|^2 = 0$. Thus, almost surely
$\lim_{t\to\infty}\zb_t = \zb_\infty = [\xb_\infty,\blab_\infty] \in \{[\xb,\blab]:\; F(\xb,\blab) = 0\}$. In particular, $\xb_\infty$ is a GNE in $\Gamma$. 
\end{proof}
\begin{rem}
 Recall that, in zero-order optimization, the minimax lower bound for minimizing a $\mu$-strongly convex function is $\Omega(1/t)$~\cite{duchi2015optimal}.
Theorem~\ref{th:main} shows that, by leveraging the equivalence between the GNE computation problem and a convex optimization problem satisfying the PL condition, one obtains an $O(1/t)$ convergence rate for quadratic games with linearly coupled equality constraints.
Since every $\mu$-strongly convex function satisfies the PL condition, the strongly convex class is strictly contained in the PL class. Consequently, the obtained $O(1/t)$ rate matches the known minimax lower bound (up to dimension-dependent constants) and is therefore asymptotically optimal for this broader problem class.
\end{rem}

\begin{rem}
Theorem~\ref{th:main} establishes convergence to a generalized Nash equilibrium (GNE)
in quadratic, generally non-monotone games with individual coupling linear
\emph{equality} constraints.
It is worth noting that this result cannot be directly extended to the case of
\emph{inequality} constraints.
Indeed, if the equality constraints $A_i x - b_i = 0$ are replaced by the
inequality constraints $A_i x - b_i \le 0$, 
the optimality conditions in Proposition~\ref{prop:1} are reformulated as
follows (see again~\cite[Theorem~8]{facchinei2007generalized}):
\begin{align}\label{eq:GNE_ineq}
    &\nabla_{\xb^i} J_i(\xb^*) + A_{i(i,:)}^\top \blab^{*,i} = \boldsymbol{0}, \\
    &(\blab^{*,i})^\top (A_i \xb^* - \bb_i) = 0, \quad \forall i\in [n].
\end{align}
The corresponding optimization problem can be, thus, formulated as
\begin{align}\label{eq:NP}
    &\min_{\xb,\, \blab}\;
        \sum_{i=1}^n
        \big\|
            \nabla_{\xb^i} J_i(\xb)
            + \rb_i^{i}
            + A_{i(i,:)}^\top \blab_i
        \big\|^2, \cr
    &\text{s.t.}\quad
        (\blab^{i})^\top (A_i \xb - \bb_i) = 0,
        \quad \forall i \in [n].
\end{align}
The distributed structure of this formulation is preserved,
and the zero-order gradient estimation methods proposed in this work
can still be applied.
However, the resulting problem is no longer convex.
The optimization problem~\eqref{eq:NP} falls under the class of
\emph{Linear Complementarity Quadratic Programs (LCQPs)},
for which convergence can, in general, only be guaranteed to a
stationary point~\cite{Hall2021SequentialConvexProgrammingLCQP}.
\end{rem}

\section{A numerical example}
In this section, we demonstrate the efficacy of the proposed Algorithm~$1$ with an example. Consider a GNE problem in the \textit{non-monotone} game $\Gamma(2,\{C_1, C_2\},\{J_1, J_2\})$ with  cost functions $J_i(\xb^i,\xb^{-i}) = \frac{1}{2}\xb^TQ_i\xb + \rb_i^T\xb + k_i, \; \forall i \in [2]$ and the local coupled sets $C_i$ defined by the linear equality constraints $A_i\xb-\bb_i = \boldsymbol{0}, \; \forall i \in [2]$, where 
\begin{align*}
    & Q_1 = \begin{bmatrix}
        7 & 1 & 1 & 0 \\
        1 & 7 & 0 & 1 \\
        1 & 0 & 7 & 1 \\
        0 & 1 & 1 & 7
    \end{bmatrix}, \quad 
    r_1 = \begin{bmatrix}
       -12 \\
       -19 \\
       -26 \\
       -33
    \end{bmatrix}, \quad 
    k_1 = 0, \\ \\
    & A_1 = \begin{bmatrix} 
            1 & 0 & 1 & 0\\
            1 & 1 & 0 & 0
        \end{bmatrix}, \quad
    b_1 = \begin{bmatrix}
        4 \\
        3
    \end{bmatrix}, \\ \\
    & Q_2 = \begin{bmatrix}
        7 & 0 & -3.5 & 1 \\
        0 & 7 & 1 & 0 \\
        -3.5 & 1 & 7 & 0 \\
        1 & 0 & 0 & 7
    \end{bmatrix}, \quad 
    r_2 = \begin{bmatrix} 
        -8 \\
       -17 \\
       -17 \\
       -29
    \end{bmatrix}, \quad
    k_2 = 0, \\ \\
    & A_2 = \begin{bmatrix} 
            1 & 1 & 1 & 1
        \end{bmatrix}, \quad 
    b_2 = 10.
\end{align*}
The GNE of this game exists and is given by $x^\ast = (1,2,3,4)^\top $. Together with $\lambda^\ast = (0,0,0)^\top$, it satisfies~\eqref{eq:GNE}. Therefore, we are in a setting in which Assumptions~\ref{assum:standing} and \ref{assum:exist} hold. 
To search for this GNE, Algorithm~$1$ was implemented with $T = 10000$ iterations, $\sigma = 0.05$ and $\delta = 0.05$. The iterates $\boldsymbol{x}(t)$ and $\boldsymbol{\lambda}(t)$ were initialized at zero. We set up the following step sizes\footnote{Individual step sizes of the same order in respect to $t$ are justified, see \cite{tat_kam_TAC} for further details.} of the order $O(1/t)$ for the components of $\boldsymbol{x}$ and $\boldsymbol{\lambda}$:  
$\gamma_{x_1^1}(t) = 0.006/(t+500)$, $\gamma_{x_1^2}(t) = 0.005/(t+500)$, $\gamma_{x_2^1}(t) = 0.015/(t+500)$, $\gamma_{x_2^2}(t) = 0.009/(t+500)$, $\gamma_{\lambda^1_1}(t) =\gamma_{\lambda^2_1}(t)= \gamma_{\lambda^1_2}(t) = 0.001/(t+1000)$. 
For the centralized first-order/gradient-based approach, a constant learning rate of $\gamma_x(t) = 0.001$ and $\gamma_\lambda(t) = 0.005$ was chosen.  As depicted in Figure \ref{Fig1}, both methods converge the GNE as time runs. The theoretic investigations are supported by the simulations: the first-order method exhibits exponential convergence (see Remark~\ref{rem:1oa}), whereas the zero-order procedure, as expected, converges more slowly and displays an $O(1/t)$ rate (Theorem~\ref{th:main}). 


\begin{figure}[h]
\centering
\includegraphics[width=0.45\textwidth]{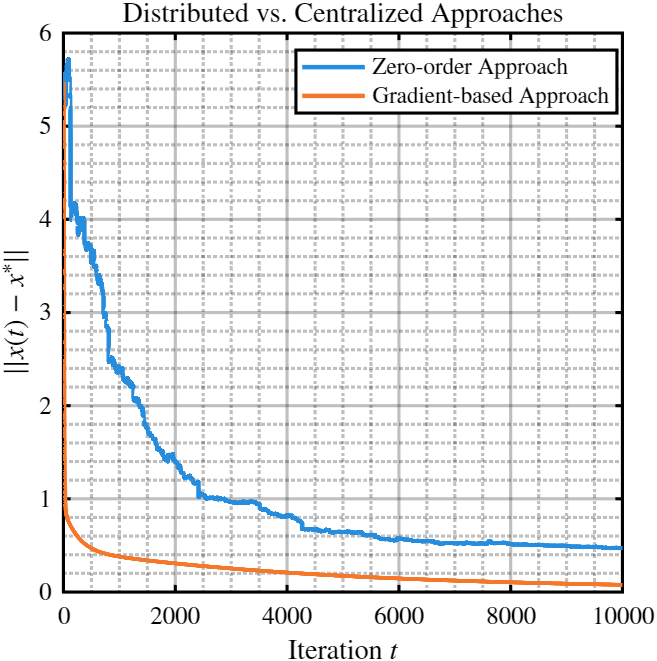} 
\caption{Convergence of Algorithm~$1$}
\label{Fig1}
\end{figure}

\section{Conclusion}
We investigated a class of generalized Nash equilibrium problems with quadratic costs and individual linear equality constraints. By reformulating the KKT conditions into a convex optimization problem with a PL-structured objective, we enabled a distributed first-order method that achieves global linear convergence to a GNE without requiring strong monotonicity or shared constraints. For scenarios where gradient information is unavailable, we proposed a fully distributed zero-order scheme based solely on local cost evaluations and constraint residuals. Future work will explore extending the continuous reformulation of GNE problems to games with general inequality constraints.
\appendix

\subsection{Robbins--Siegmund convergence theorem}\label{app:RS}
\begin{lem}[Lemma 11 in \cite{polyak}]\label{lem:RS}
Let $(V_t)_{t\ge 0}$ be a sequence of nonnegative random variables generating 
sigma-algebras  $(\mathcal F_t)_{t\ge 0}$. 
Suppose there exist sequences of nonnegative random variables 
$\{a_t\}$, $\{b_t\}$, $\{c_t\}$ such that almost surely
\[
\E\{V_{t+1}\mid \EuScript F_t\} \;\le\; (1-a_t) V_t + b_t + c_t,
\qquad t\ge 0,
\]
with
$\sum_{t=0}^\infty a_t = \infty$, $\sum_{t=0}^\infty b_t < \infty$, and  
$\sum_{t=0}^\infty c_t < \infty$.
Then  $V_t \to 0$ almost surely.
\end{lem}

\subsection{Chung's lemma}\label{app:Chung}
\begin{lem}[Lemma 4 in \cite{polyak}]\label{lem:chung}
Suppose that $\{u_t\}$, $t\ge1$, is a sequence of real numbers such that 
$u_{t+1}\le\left(1-\frac{c}{t}\right)u_t+\frac{d}{t^{p+1}}$, for $t\ge n_0$,
where $c>0$, $d>0$, $p>0$. Then, if $c>p$,
$u_t = O\left(\frac{1}{t^p}\right)$.
\end{lem}

\subsection{Martingale convergence theorem}\label{app:mart}
\begin{lem}[Lemma 2.4.1 in \cite{Kushner}]\label{lem:mart}
Let $(\mathcal F_t)_{t\ge 0}$ be a filtration of sigma algebras, i.e. $\mathcal F_t\subseteq \mathcal{F}_{t+1}$ for all $t\ge 0$, and let $(\zeta_t)_{t\ge 1}$ be a martingale–difference sequence:
$\mathbb E[\zeta_t\mid \mathcal F_{t-1}]=0$ almost surely, with finite conditional second moments
$\mathbb E[\zeta_t^2\mid \mathcal F_{t-1}]<\infty$ almost surely. 
Let $(a_t)_{t\ge 1}$ be a real sequence such that
$\sum_{t=1}^\infty a_t^2 \;<\;\infty$.
Then the series $\sum_{t=1}^\infty a_t \zeta_t<\infty$ almost surely.
\end{lem}

\subsection{More details on~\eqref{eq:zeta1}}\label{app:details}
We rewrite~\eqref{eq:zeta} using the next terms: 
\begin{align*}
    \zeta_{\xb^i}
=&\sum_{j=1}^n\!\underbrace{\langle\nb_j,\eta\rangle\langle\pb_j,\eta\rangle(\|\eta^j\|^2-d)}_{=:~T_{1,j}}\;\xi_x^i\cr
&+\frac{\delta}{2}\sum_{j=1}^n\!\underbrace{\langle\pb_j,\eta\rangle^2(\|\eta^j\|^2-d)}_{=:~T_{2,j}}\;\xi_x^i
+\sum_{j=1}^n\!\underbrace{\langle\qb_j,\eta\rangle\eta^i}_{=:~T_{3,j}},
\end{align*}
To obtain~\eqref{eq:zeta1}, we use the following fact: if $u\sim\mathcal N(0,I_{nd})$ and $P_j=e_j e_j^\top\otimes I_d$ is the projector onto the $j$-th $d$-block, then for any $a,b\in\mathbb{R}^{nd}$,
\begin{align}
\mathbb{E}\!\big[\langle a,u\rangle\,\langle b,u\rangle\,\|u^j\|^2\big]
&= a^\top(dI_{nd}+2P_j)b, \label{eq:key1}\\
\mathbb{E}\!\big[\langle a,u\rangle\,\langle b,u\rangle\,(\|u^j\|^2-d)\big]
&= 2\,a^\top P_j b, \label{eq:key2}\\
\mathbb{E}\!\big[\langle b,u\rangle^2\,(\|u^j\|^2-d)\big]
&= 2\,b^\top P_j b \;=\; 2\,\|P_j b\|^2, \label{eq:key3}\\
\mathbb{E}\!\big[\langle q,u\rangle\,u^i\big]
&= P_i q. \label{eq:key4}
\end{align}
Next, we condition $T_{1,j}$ on $(\xi_x,\xi_\lambda)$, so that $(\nb_j,\pb_j,\qb_j,\xi_x^i)$ are deterministic with respect to the
inner Gaussian $\eta$. Then $
\mathbb{E}\big[T_{1,j}\mid \xi_x,\xi_\lambda\big]=2\,\langle\nb_j,P_j\pb_j\rangle
$ for each $j$.
Multiplying by $\xi_x^i$ and summing over $j$ gives
\[
\sum_{j=1}^n \mathbb{E}[T_{1,j}\mid \xi_x,\xi_\lambda]\;\xi_x^i
=2\sum_{j=1}^n\langle\nb_j,P_j\pb_j\rangle\,\xi_x^i.
\]
 Applying \eqref{eq:key3} to the second sum of $T_{2,j}$, we obtain
$\mathbb{E}\big[\langle\pb_j,\eta\rangle^2(\|\eta^j\|^2-d)\big]
=2\,\|\!P_j\pb_j\|^2$.
Hence,
\[
\frac{\delta}{2}\sum_{j=1}^n \mathbb{E}[T_{2,j}\mid \xi_x,\xi_\lambda]\;\xi_x^i
=\delta\sum_{j=1}^n \|P_j\pb_j\|^2\;\xi_x^i = \mathrm{Pol}_2(\xi_x,\xi_\lambda) \xi_x^i.
\]
Since $\pb_j$ is \emph{linear} in $(\xi_x,\xi_\lambda)$ (it arises from the outer shift of the affine
gradient), the map $(\xi_x,\xi_\lambda)\mapsto \sum_j\|P_j\pb_j\|^2$ is a homogeneous quadratic polynomial,
which we denote by $\mathrm{Pol}_2(\xi_x,\xi_\lambda)$.

Finally, applying \eqref{eq:key4} to the third sum  of $T_{3,j}$, we get
\[
\mathbb{E}\big[T_{3,j}\mid \xi_x,\xi_\lambda\big]
=\mathbb{E}\big[\langle\qb_j,\eta\rangle\,\eta^i\big]
=P_i \qb_j.
\]
Summing over $j$ yields $\sum_{j=1}^n P_i\qb_j$, which equals
$\nabla_{\xb^i}\sum_{j=1}^n c_j(\xb)$ by the definition of $\qb_j$.
Thus, adding the three contributions above, we conclude the desired expression~\eqref{eq:zeta1}.
\bibliographystyle{plain}
\bibliography{srtrMonGames_ref}
\end{document}